\documentclass[11pt]{article}
\usepackage{fullpage,url,natbib}
\usepackage{amsthm,amsfonts,amsmath,amssymb,epsfig,color,float,graphicx,verbatim}
\usepackage{hyperref}
\hypersetup{
	colorlinks   = true, %
	urlcolor     = blue, %
	linkcolor    = blue, %
	citecolor   = blue %
}

\newtheorem{theorem}{Theorem}

\newtheorem*{proposition*}{Proposition}
\newtheorem{lemma}{Lemma}

\renewcommand{\eqref}[1]{Eq.~(\ref{#1})}

\newcommand{\R}{\mathbb{R}}

\newcommand{\N}{\mathbb{N}}

\newcommand{\mtv}{\boldsymbol{\delta}}

\newcommand{\inv}{^{-1}}

\newcommand{\abs}[1]{\left| #1 \right|}
\newcommand{\paren}[1]{\left( #1 \right)}
\newcommand{\sqprn}[1]{\left[ #1 \right]}
\newcommand{\nrm}[1]{\left\Vert #1 \right\Vert}

\newcommand{\vertiii}[1]{{\left\vert\kern-0.25ex\left\vert\kern-0.25ex\left\vert #1 
    \right\vert\kern-0.25ex\right\vert\kern-0.25ex\right\vert}}
\newcommand{\tv}[1]{\nrm{#1}_{\textup{\tiny\textsf{TV}}}}

\newcommand{\mexp}{\mathbb{E}}

\newcommand{\E}{\mathop{\mexp}}
\renewcommand{\P}{\mathbb{P}}

\newcommand{\eps}{\varepsilon}

\newcommand{\kl}[2]{D_{
\mathrm{KL}
  }
  \left(#1 || #2\right)
  }

\usepackage{xcolor}

\newcommand{\beq}{\begin{eqnarray*}}
\newcommand{\eeq}{\end{eqnarray*}}
\newcommand{\beqn}{\begin{eqnarray}}
\newcommand{\eeqn}{\end{eqnarray}}

\newcommand{\del}{\partial}

\newcommand{\ddel}[2]{\frac{\del#1}{\del#2}}

\newcommand{\maavar}[2]{ 
\overset{\textup{(#1)}}{#2}
}

\usepackage{ifmtarg}
\usepackage{xifthen}%
\newcommand{\ent}[1][]{%
\ifthenelse{\isempty{#1}}{%
\mathrm{H}
}{
\mathrm{H}^{(#1)}
}}

\newcommand{\loch}[1][]{%
\ifthenelse{\isempty{#1}}{%
\mathrm{h}
}{
\mathrm{h}^{(#1)}
}}

\newcommand{\mathe}{\mathrm{e}}
\newcommand{\mathd}{\mathrm{d}}

\newcommand{\hide}[1]{}

\newcommand{\ub}{\operatorname{UB}}
\newcommand{\lb}{\operatorname{LB}}

\newcommand{\set}[1]{\left\{ #1 \right\}}

\newcommand{\Ber}{\operatorname{Ber}}

\newtheorem*{rep@theorem}{\rep@title}
\newcommand{\newreptheorem}[2]{%
\newenvironment{rep#1}[1]{%
 \def\rep@title{#2 \ref{##1}}%
 \begin{rep@theorem}}%
 {\end{rep@theorem}}}
\makeatother

\newreptheorem{proposition}{Proposition}

\renewcommand{\eqref}[1]{(\ref{#1})}

\title{
On the tensorization of 
the variational distance
}

\author{%
  Aryeh Kontorovich \\
  \texttt{karyeh@cs.bgu.ac.il}
  }

\date{}

\begin{document}

\maketitle
\begin{center}\vspace{-1cm}\today\vspace{0.5cm}\end{center}

\begin{abstract}
If one seeks to estimate the total variation
between
two product measures
$
\tv{
P^\otimes_{1:n}
-
Q^\otimes_{1:n}
}
$
in terms of their
marginal TV sequence
$\mtv=
\paren{
\tv{P_1-Q_1}
,
\tv{P_2-Q_2}
,
\ldots
,
\tv{P_n-Q_n}
}
$,
then trivial upper and lower bounds are provided by
$
\nrm{\mtv}_\infty
\le
\tv{
P^\otimes_{1:n}
-
Q^\otimes_{1:n}
}
\le
\nrm{\mtv}_1
$.
We improve the lower bound to
$
\nrm{\mtv}_2
\lesssim
\tv{
P^\otimes_{1:n}
-
Q^\otimes_{1:n}
}
$,
thereby reducing the gap between the upper and lower bounds
from $\sim n$ to $\sim\sqrt n$.
Furthermore, we show that {\em any} estimate
on
$
\tv{
P^\otimes_{1:n}
-
Q^\otimes_{1:n}
}
$
expressed in terms of $\mtv$
must necessarily exhibit a gap of $\sim\sqrt n$
between the upper and lower bounds in the worst case, establishing a sense 
in which our estimate is optimal.
Finally, we identify a natural class of distributions
for which $\nrm{\mtv}_2$ approximates
the TV distance up to absolute multiplicative constants.

\end{abstract}

\section{Introduction}
Combining several probability distributions into a single
product measure is one of the most basic operations in probability
theory. 
Among the myriad notions of distance on distributions,
the total variation metric occupies an amply motivated central role \citep{MR1843146,gibbs02}. It is therefore
of fundamental importance to understand the basic ways in which
these two objects interact: how do product measures behave
under the variational metric? Perhaps surprisingly, our
understanding of this basic problem is still somewhat rudimentary
and although considerable progress has been made, natural and
challenging
open questions remain.

We will restrict our discussion to distributions on finite sets,
since for the kind of problems we study, these already capture the general case \citep[Theorem 7.6]{Polyanskiy_Wu_2024}.
If $P,P'$ are distributions on $\Omega,\Omega'$,
then $P\otimes P'$ is the product measure on $\Omega\times\Omega'$
given by $
(P\otimes P')(\omega,\omega)=P(\omega)P'(\omega')
$. For families of distributions $P_1,\ldots,P_n$,
we 
write
$
P^\otimes_{1:n}
=
P_1\otimes P_2\otimes\ldots\otimes P_n
$.
Our central object of interest is
the total variation distance between two product measures:
$
\tv{
P^\otimes_{1:n}
-
Q^\otimes_{1:n}
}
$.
The form
$
\tv{\cdot}
=
\frac12\nrm{\cdot}_1
$
(see \eqref{eq:tv-def})
suggests
that the complexity of computing
this
distance scales linearly with the support size.
However, in the interesting case
of product measures,
the 
support size
will grow exponentially with $n$,
and hence direct computation will be intractable. 
\citet[p. 124]{Polyanskiy_Wu_2024}
note that
``computing the total variation between two $n$-fold product distributions is typically difficult'', which motivates
various approximations by surrogate distances; more on this below.

Turning for the moment to algorithmic aspects,
\citet{DBLP:conf/ijcai/0001GMMPV23}
showed 
even for support size $2$,
the general problem of computing the variational distance
between two $n$-fold product measures exactly
is
computationally hard
in the $\#$P-complete sense.
\citet{FengApproxTV23} gave an efficient randomized algorithm for obtaining a $1\pm\eps$
multiplicative approximation with confidence $\delta$, in time $O(\frac{n^2}{\eps^2}\log\frac1\delta)$,
which was later derandomized by 
\citet{Feng24Deterministically}.
A novel connection between TV distance estimation and probabilistic inference was recently established by \citet{bhattacharyya2024total}.

\paragraph{Notation.}
%We will need a few basic definitions.
$\Omega$ will always denote a finite set.
We will lighten notation by writing $P(\omega)=P(\set{\omega})$
for 
distributions 
$P$ on $\Omega$
and
$\omega\in\Omega$.
For two distributions $P,Q$ on 
$\Omega$, their total variation
distance is defined by
\beqn
\label{eq:tv-def}
\tv{P-Q}
&=&
\frac12\sum_{\omega\in\Omega}|P(\omega)-Q(\omega)|
.
\eeqn
We write $[n]:=\set{1,\ldots,n}$
and use
standard vector norm notation
$\nrm{x}_r^r=\sum_{i\in[n]}|x_i|^r$
and $\nrm{x}_\infty=\max_{i\in[n]}|x_i|$
for $x\in\R^n$.
For $p\in[0,1]$, 
$\Ber(p)$
denotes
the Bernoulli measure on $\set{0,1}$:
$\Ber(p)(0)=1- p$
and
$\Ber(p)(1)=p$.
For $n\in\N$ and $p=(p_1,\ldots,p_n)\in[0,1]^n$, 
$\Ber(p)$ denotes
the product of $n$ Bernoulli distributions with parameters $p_i$:
\beqn
\label{eq:berp}
\Ber(p) &:=& \Ber(p_1) \otimes \Ber(p_2)\otimes \ldots \otimes \Ber(p_n).
\eeqn

For two families of distributions,
$P_{i\in[n]}$
and
$Q_{i\in[n]}$,
we define their marginal total variation 
sequence
by
\beqn
\label{eq:mtv}
\mtv(
P_{i\in[n]}
,
Q_{i\in[n]}
)
&=&
\paren{
\tv{P_1-Q_1}
,
\tv{P_2-Q_2}
,
\ldots
,
\tv{P_n-Q_n}
}
\in\R^n.
\eeqn

{\em Tensorization} generally refers to
analyzing a complex, high-dimensional object in terms of its
simpler
low-dimensional components. In our context, it means
estimating
$
\tv{
P^\otimes_{1:n}
-
Q^\otimes_{1:n}
}
$
in terms of the 
marginal TV sequence
$
\mtv=
\mtv(
P_{i\in[n]}
,
Q_{i\in[n]}
)
$.
A trivial bound of this form is
\beqn
\label{eq:triv}
\nrm{\mtv}_\infty
\;\le\;
\tv{
P^\otimes_{1:n}
-
Q^\otimes_{1:n}
}
\;\le\;
\min\set{1,\nrm{\mtv}_1}
,
\eeqn
where the lower bound follows from the
data processing inequality (Lemma~\ref{lem:data-proc})
and the upper bound is an immediate consequence of
the coupling 
estimate
$
\tv{P\otimes P'
-
Q\otimes Q'
}
\le
\tv{P-Q}
+
\tv{P'-Q'}
-
\tv{P-Q}
\tv{P'-Q'}
$
\citep[Lemma 2.2]{kontorovich12}.
Each of the bounds in \eqref{eq:triv} is tight in the sense
of being achievable, but there is a factor of $n$ gap between the upper and lower.

Our main result is the following sharpening of \eqref{eq:triv}:
\begin{theorem}
\label{thm:main}
For two families of distributions,
$P_{i\in[n]}$
and
$Q_{i\in[n]}$,
and
their marginal total variation 
sequence
$
\mtv=\mtv(
P_{i\in[n]}
,
Q_{i\in[n]}
)
$
defined in
\eqref{eq:mtv},
we have
\beq
\tv{
P^\otimes_{1:n}
-
Q^\otimes_{1:n}
}
&\ge&
c\min\set{1,\nrm{\mtv}_2},
\eeq
where $
c>
%0.2889
0.1798
$
is a universal constant.
\end{theorem}
Since $\nrm{\mtv}_1\le\sqrt n\nrm{\mtv}_2$,
Theorem~\ref{thm:main} improves the gap between the upper
and lower bounds in \eqref{eq:triv} from $n$ to $\sqrt n$.
It turns out that this is the best one can do in terms of
the 
marginal total variation sequence:
\begin{theorem}
\label{thm:sqrt-opt}
For every $n\ge1$, there exist
$p,p',q,q'\in[0,1]^n$
such that
$\nrm{\mtv}_2=
\nrm{p-q}_2
=
\nrm{p'-q'}_2
$,
while
\beq
\frac
{
\tv{
\Ber(p)-
\Ber(q)
}
}
{
\tv{
\Ber(p')-
\Ber(q')
}
}
&\ge&
c\sqrt n,
\eeq
where $c>0$ is an absolute constant.
\end{theorem}

We are able to pinpoint the precise step in the proof
where
the $\sqrt n$ gap between the upper and lower bounds
is introduced: namely, in the use of symmetrization
(Lemma~\ref{lem:sym})
to reduce
the analysis of
$
\tv{
\Ber(p)-\Ber(q)
}
$
from the case of general $p,q\in[0,1]^n$
to the symmetric special case where $p=1-q$.
As our next result shows, in the symmetric case,
the gap disappears:
\begin{theorem}
\label{thm:symm-ub}
For $p\in[0,1]^n$
and $q=1-p$, we have
\beq
\tv{
\Ber(p)-\Ber(q)
}
&\le&
\nrm{p-q}_2.
\eeq
\end{theorem}

\paragraph{Open problem.}
In light of the hardness result of
\citet{DBLP:conf/ijcai/0001GMMPV23},
we do not expect a simple, analytically tractable
formula 
(or even just an efficient algorithm)
for computing
$\tv{
\Ber(p)-
\Ber(q)
}
$
exactly.
An efficient algorithm for approximating TV to arbitrary
precision was recently obtained by \citet{Feng24Deterministically}.
%\citet{FengApproxTV23} 
%give reason to believe 
%that an arbitrary-accuracy
%approximation should be achievable via
%an efficient deterministic algorithm.
We ask whether there are simple, analytically tractable
upper and lower bounds $\ub(p,q), \lb(p,q)$
whose ratio is bounded by some absolute constant independent
of $n$, $p$, and $q$.

\paragraph{Related work.}
\citet{gibbs02} discuss a number of important distance functions over probability measures
commonly used in statistics and probability
and their relation to the total variation metric.
In 
information theory and
statistics, 
KL-divergence and the Hellinger distance
are frequently used as surrogates for the latter. 
These are defined, respectively, by
\beq
\kl{P}{Q} &=&
\sum_{\omega\in\Omega}
P(\omega)\log\frac
{ P(\omega)}
{ Q(\omega)},
\\
H^2(P,Q) &=& 
\sum_{\omega\in\Omega}
\paren{\sqrt{P(\omega)}-\sqrt{Q(\omega)}}^2,
\eeq
and feature the convenient property of tensorizing,
in the sense of 
$
\kl{
P^\otimes_{1:n}
}{
Q^\otimes_{1:n}
}
$
and
$
H^2(
P^\otimes_{1:n}
,
Q^\otimes_{1:n}
)
$
both being easily 
expressible
in terms of
the component
$
\kl{
P_i
}{
Q_i
}
$
and
$
H^2(
P_i
,
Q_i
)
$,
respectively. The following relations between these
distance measures and $\tv{\cdot}$
are known \citep[Eq.((7.22)]{Polyanskiy_Wu_2024}:
\beqn
\label{eq:H2}
\frac12 H^2(P,Q)
\;\le\;
\tv{P-Q}
\;\le\;
H(P,Q)\sqrt{1-\frac{H^2(P,Q)}4}
,
\eeqn
and
\citep[Lemma 3.10]{Even-DarKM07},
\citep[Theorem 7.10]{Polyanskiy_Wu_2024}:
\beqn
\label{eq:kl}
\frac{\kl{P}{Q}}{
2
\log P_{\min}
}
\;\le\;
\tv{P-Q}
\;\le\;
\sqrt{
\frac{\kl{P}{Q}}{
2
}
},
\eeqn
where
$
P_{\min}
=
\min_{\omega\in\Omega} P(\omega)
$
and $P_{\min}<\frac12$ is assumed.
See
\citet{SasonV16}
for additional relationships between the various $f$-divergences,
of which 
$D_{\mathrm{KL}}$
and $\tv{\cdot}$
are special cases,
and
\citet{CohenKKW23}
for bounds involving $\tv{\cdot}$ and moments of information.

Although the bounds
in \eqref{eq:H2} and \eqref{eq:kl}
are useful for various applications,
there is a clear
sense in which
these are inferior to
\eqref{eq:triv}. 
If we express generic lower and upper bounds as
$\lb(P,Q)\le\tv{P-Q}\le\ub(P,Q)$,
we say that these have the bounded-ratio property if
\beqn
\label{eq:brp}
\sup_{P,Q}\frac{\ub(P,Q)}{\lb(P,Q)} < \infty.
\eeqn
Then \eqref{eq:triv} satisfies this property with a ratio bound of $n$,
while neither of (\ref{eq:H2}), (\ref{eq:kl})
does.

\section{Proofs}

\subsection{
Scheff\'e's identity,
data processing, reduction to Bernoulli}

We will make 
repeated 
use of 
Scheff\'e's identity
\citep[Lemma 2.1]{tsybakov09}:
\beqn
\label{eq:scheffe}
\tv{P-Q}
\;=\;
\max_{A\subset\Omega}P(A)-Q(A)
\;=\;
1-
\nrm{\min\set{P,Q}}_1
,
\eeqn
where 
$\nrm{\min\set{P,Q}}_1
=\sum_{\omega\in\Omega}
\min\set{P(\omega),Q(\omega)}
$.

A general form of the data processing inequality for $f$-divergences
may be found in 
\citet[Theorem 7.4]{Polyanskiy_Wu_2024}. Since total variation
is a special case (Eq. (7.3) ibid.), we restrict our statement to TV.
\begin{lemma}[Data processing inequality]
\label{lem:data-proc}
Let $T$ be a Markov operator mapping distributions to distributions.
Then, for all distributions $P,Q$, we have
\beq
\tv{T(P)-T(Q)}
&\le&
\tv{P-Q}.
\eeq
\end{lemma}

\begin{lemma}
\label{lem:red-ber}
If Theorem~\ref{thm:main} holds for 
the special case
where
$P=\Ber(p)$, $Q=\Ber(q)$
for 
%some
$p,q\in[0,1]^n$,
then it 
holds in the general case.
\end{lemma}
\begin{proof}
Let $P_i$ and $Q_i$ be distributions over $\Omega_i$,
$i\in[n]$.
For each $i\in[n]$, define
$$A_i=\set{
\omega\in\Omega_i:
P_i(\omega)>Q_i(\omega)
}$$
and
$p_i=P_i(A_i)$,
$q_i=Q_i(A_i)$.
Clearly, there is a Markov operator $T_i$ that maps
$P_i$ and $Q_i$ to $\Ber(p_i)$ and $\Ber(q_i)$, respectively
(by collapsing the states in $A_i$ into a single state and doing the
same for the complement). Applying all of the $T_i$s simultaneously yields
a Markov operator $T$ mapping 
$P^\otimes_{1:n}$
and
$Q^\otimes_{1:n}$
to
$\Ber(p)$
and
$\Ber(q)$,
respectively, where
$p=(p_1,\ldots,p_n)$
and
$q=(q_1,\ldots,q_n)$. It follows from
Lemma~\ref{lem:data-proc} that
$
\tv{
P^\otimes_{1:n}
-
Q^\otimes_{1:n}
}
\ge
\tv{
\Ber(p)
-
\Ber(q)
}
$.
On the other hand,
by
\eqref{eq:scheffe},
$\tv{P_i-Q_i}=|p_i-q_i|$,
and so
$
\mtv(
P^\otimes_{1:n}
,
Q^\otimes_{1:n}
)
=
\mtv(
\Ber(p)
,
\Ber(p)
)
=
(|p_1-q_1|,\ldots,
|p_n-q_n|)
$,
where $\mtv$
is
the 
marginal total variation sequence defined 
in \eqref{eq:mtv}.

\end{proof}

\subsection{
Symmetrization
}

A key step in our proof
reduces the problem of lower-bounding
$
\tv{
\Ber(p)
-
\Ber(q)
}
$
in terms of $\nrm{p-q}_2$
for general $p,q\in[0,1]^n$
to the 
{\em symmetric}
special case
where $p_i=1-q_i$,
incurring a loss of
only a factor of $2$.

\begin{lemma}[symmetrization \citep{thomas-mo-24}]
\label{lem:sym}
For $p,q\in[0,1]^n$, define 
\beqn
\label{eq:gpq}
\hat\gamma_i=
\frac{|p_i-q_i|}{1+|p_i+q_i-1|},
\quad
\hat p_i=\frac12+\frac{\hat\gamma_i}2,
\quad
\hat q_i=\frac12-\frac{\hat\gamma_i}2,
\qquad
i\in[n]
.
\eeqn
Then
\beqn
\label{eq:pqhatell1}
\abs{\hat p_i-\hat q_i} &\ge& \frac12\abs{p_i-q_i},
\qquad i\in[n]
;\\
\label{eq:pqhattv}
\tv{
\Ber(p)
-
\Ber(q)
}
&\ge&
\tv{
\Ber(\hat p)
-
\Ber(\hat q)
}
.
\eeqn
\end{lemma}
\begin{proof}
There is no loss of generality in assuming $p_i>q_i$,
since the $p_i=q_i$ case is trivial, and if $p_i<q_i$,
we can apply the transformation $p_i\mapsto 1-p_i$, $q_i\mapsto1-q_i$,
which leaves both $|p_i-q_i|$
and
$
\tv{
\Ber(p)
-
\Ber(q)
}
$ unchanged.
We will use the notation $X\sim P$
to denote a random variable $X$ distributed according to the law $P$.
We proceed to construct a 
randomized
function $F_{p,q}:\set{0,1}\to\set{0,1}$
with the following property: for any $p,q\in[0,1]$,
if 
$X\sim\Ber(p)$ 
then
$
F_{p,q}(X)\sim
\Ber(\frac12+\frac\gamma2)
$ 
and
if 
$Y\sim\Ber(q)$ 
then
$
F_{p,q}(Y)\sim\Ber(\frac12-\frac\gamma2)$,
where
%\beqn
%\label{eq:gamdef}
$
\gamma=
\frac{|p-q|}{1+|p+q-1|}
.
$
%\eeqn
We define $F_{p,q}$
by its stochastic matrix
\beqn
\label{eq:M}
M = 
\left(\begin{array}{ll} 
\P(F_{p,q}(1)=1) & 
\P(F_{p,q}(1)=0) \\ 
\P(F_{p,q}(0)=1) & 
\P(F_{p,q}(0)=0) 
\end{array}\right) 
= 
\left(\begin{array}{ll}
\frac12 + \frac{\gamma}{p-q}\frac{2-p-q}{2} & 
\frac12 - \frac{\gamma}{p-q}\frac{2-p-q}{2} \\ 
\frac12 - \frac{\gamma}{p-q}\frac{p+q}{2} & 
\frac12 + \frac{\gamma}{p-q}\frac{p+q}{2} 
\end{array}\right).
\eeqn
To verify that $F_{p,q}$ is a valid randomized map, we must check that $M$ is indeed a stochastic matrix. Its rows obviously sum to $1$ and its entries are non-negatives, since 
\beq
0 &\le& \frac{\gamma}{p-q}\frac{2-p-q}{2} = \frac12 \frac{2-p-q}{1+|1-p-q|} \le \frac12 ;\\
0 &\le& \frac{\gamma}{p-q} \frac{p+q}{2} = \frac12 \frac{p+q}{1+|p+q-1|} \le \frac12.
\eeq
Next, we verify 
that for $X\sim\Ber(p)$,
$F_{p,q}(X)\sim\Ber(\frac12+\frac\gamma2)$,
which boils down to a matrix-vector multiplication:
\beq
\big(
\P(F(X)=1), \P(F(X)=0)
\big) 
&=& 
\big(\P(X=1) , \P(X=0)\big) M
\\
&=& 
\big(p , 1-p\big) M
\\&=&
\left( \frac{p+1-p}{2} + \frac{\gamma}{2(p-q)}\big(p(2-p-q) -(1-p)(p+q)\big) \right.,
\\&&
\left.
\frac{p+1-p}{2}  - \frac{\gamma}{2(p-q)}\big(p(2-p-q)-(1-p)(p+q)\big) 
\right)
\\
&=&
 \left( \frac12 + \frac{\gamma}{2}, \frac12 - \frac{\gamma}{2}\right).
\eeq
A similar calculation shows that 
$F_{p,q}(Y)\sim\Ber(\frac12-\frac\gamma2)$
for
$Y\sim\Ber(q)$, as required.

We now extend 
$F_{p,q}$
from a random map $\set{0,1}\mapsto\set{0,1}$
determined by
$p,q\in[0,1]$
to a random map $\set{0,1}^n\mapsto\set{0,1}^n$
determined by
$p,q\in[0,1]^n$
in the natural way.
Namely, 
for
$p,q\in[0,1]^n$,
if 
$X=(X_1,\ldots,X_n)\sim\Ber(p)$
and
$Y=(Y_1,\ldots,Y_n)\sim\Ber(q)$,
we define
$\hat X=(\hat X_1,\ldots,\hat X_n)$
by 
$\hat X_i=F_{p_i,q_i}(X_i)$
and
$\hat Y=(\hat Y_1,\ldots,\hat Y_n)$
by $\hat Y_i=F_{p_i,q_i}(Y_i)$.
Since all of the coordinates are independent, we have
$\hat X\sim \Ber(\hat p)$
and
$\hat Y\sim \Ber(\hat q)$,
with $\hat p,\hat q,
$
as defined in \eqref{eq:gpq}.
Then \eqref{eq:pqhattv}
follows from
the
data processing inequality (Lemma~\ref{lem:data-proc}),
and
\eqref{eq:pqhatell1}
from the elementary fact that
\beq
\frac{|u-v|}{1+|u+v-1|}
&\ge&
\frac12
|u-v|,
\qquad u,v\in[0,1].
\eeq

\end{proof}

\subsection{Proof of Theorem~\ref{thm:main}}
The following 
numerical 
%inequalities 
relations
are elementary to verify:
\beqn
\label{eq:uvmin}
\min\set{u,v}
&=&
\sqrt{uv}\exp\paren{-\frac12\abs{\log\frac{u}{v}}},
\qquad
u,v>0,
\\
\label{eq:log2x}
\log\frac{1+x}{1-x}
&\ge&
2x,
\qquad x\ge0,
\\
\label{eq:expx}
1-\exp(-x) &\ge& 
\paren{
2-2
{\mathe}^{-1/2}
}
\min\set{x,\frac12}
,
\qquad x\ge0
%\in[0,1/2]
.
\eeqn

By Lemma~\ref{lem:red-ber},
it suffices to prove the claim for
$P=\Ber(p)$ and $Q=\Ber(q)$
with
arbitrary
$p,q\in[0,1]^n$.
Now
Lemma~\ref{lem:sym}
yields $\hat p,\hat q\in[0,1]^n$
such that
$\hat q=1-\hat p$,
$
\nrm{\hat p-\hat q}_2
\ge
\frac12\nrm{ p- q}_2
$,
and
$
\tv{
\Ber(p)
-
\Ber(q)
}
\ge
\tv{
\Ber(\hat p)
-
\Ber(\hat q)
}
$.
%We may assume without loss of generality that
%$\nrm{p-q}_2\le1$,
%since the Markov operator in \eqref{eq:M}
%can be easily modified\footnote{
%It is easily verified that
%if $\hat\gamma_i$ in \eqref{eq:gpq}
%is replaced by any $c\hat\gamma_i$
%for $c\in[0,1]$,
%then 
%\eqref{eq:pqhatell1}
%will hold with 
%$\frac{c}2$
%in place of $\frac12$
%and
%\eqref{eq:pqhattv}
%will hold unchanged.
%Thus, if $\nrm{\hat\gamma}_2\le\frac12$,
%we choose $c=1$,
%and otherwise,
%we choose $c=(2\nrm{\hat\gamma}_2)\inv$.
%}
%to ensure that
%$\nrm{\hat p-\hat q}_2\le\frac12$
%while maintaining 
%\eqref{eq:pqhatell1}
%and
%\eqref{eq:pqhattv}.
Writing 
$
\Ber(\hat p)
=\hat P=
\hat P^\otimes_{1:n}
$
and
$
\Ber(\hat q)
=\hat Q=
\hat Q^\otimes_{1:n}
$
and invoking
\eqref{eq:scheffe},
we have
\beq
1-\tv{\hat P-\hat Q}
&=&
\sum_{\omega\in\set{0,1}^n}
\min\set{\hat P(\omega), \hat Q(\omega)}
\\
&\maavar{a}=&
\sum_{\omega\in\set{0,1}^n}
\sqrt{\hat P(\omega) \hat Q(\omega)}
\exp\paren{-\frac12\abs{\log\frac{\hat P(\omega)}{\hat Q(\omega)}}}
\\
&=&
\sqrt{
\prod_{i=1}^n
\hat p_i \hat q_i
}
\sum_{\omega\in\set{0,1}^n}
\exp\paren{-\frac12\abs{\log\frac{\hat P(\omega)}{\hat Q(\omega)}}}
\\
&
\maavar{b}=
&
\prod_{i=1}^n
2
\sqrt{
\hat p_i(1- \hat p_i)
}
\cdot
\E_{Z\sim\mathrm{Uniform}\set{0,1}^n}
\exp\paren{-\frac12\abs{\log\frac{\hat P(Z)}{\hat Q(Z)}}}
\\
&\le&
\E_{Z\sim\mathrm{Uniform}\set{0,1}^n}
\exp\paren{-\frac12\abs{\log\frac{\hat P(Z)}{\hat Q(Z)}}}
,
\eeq
where 
(a) 
follows
from \eqref{eq:uvmin}
and
(b) by multiplying each of the two factors
by
$2^n$
and $2^{-n}$, respectively.
Thus,
\beq
\tv{
\hat P-\hat Q
}
&=&
1-\nrm{
\min\set{\hat P,\hat Q}
}_1\\
&\ge&
\E_Z
\sqprn{
1-
\exp\paren{-\frac12\abs{
\sum_{i=1}^n 
\log\frac{\hat P_i(Z_i)}{\hat Q_i(Z_i)}
}
}
}
\\
&
\maavar{c}{\ge}
&
0.4571
\E_Z
\sqprn{
1-
\exp\paren{
-\frac12
\sqrt{
\sum_{i=1}^n\paren{
\log\frac{\hat P_i(Z_i)}{\hat Q_i(Z_i)}
}^2
}
}
}
\\&
\maavar{d}
=
&
0.4571
\sqprn{
1-
\exp\paren{
-\frac12
\sqrt{
\sum_{i=1}^n\paren{
\log
\frac{\hat p_i}{
\hat q_i
}
}^2
}
}
}
\\&
\maavar{e}
\ge
&
0.4571
\sqprn{
1-
\exp\paren{
-\frac12
\sqrt{
\sum_{i=1}^n\paren{
2
(\hat p_i-\hat q_i)
}^2
}
}
}
\\&=&
0.4571
\sqprn{
1-
\exp\paren{
-
%2
\nrm{\hat p-\hat q}_2
}
}
\\&
\maavar{f}
\ge
&
%0.5778
0.3597
\min\set{
\nrm{\hat p-\hat q}_2,
\frac12}
,
\eeq
where (c) follows by 
applying
Lemma~\ref{lem:lowther}
to the 
concave function
$
t\mapsto1-\exp(-t/2)
$
and the
symmetric random variables 
$
\log\frac{\hat P_i(Z_i)}{\hat Q_i(Z_i)}
$,
(d) by noticing that
$
\paren{\log\frac{P_i(Z_i)}{Q_i(Z_i)}}^2
=
\paren{\log\frac{\hat p_i}{\hat q_i}}^2
$
almost surely,
(e) by applying 
\eqref{eq:log2x}
to
$
\log\frac{\hat p_i}{\hat q_i}
=
\log\frac{\frac12+\frac{\hat\gamma_i}2}{\frac12-\frac{\hat\gamma_i}2}
$
and (f)
via \eqref{eq:expx}.
Since 
$
\nrm{\hat p-\hat q}_2
\ge\frac12
\nrm{ p- q}_2
$, we are done. \qed

\subsection{Proof of Theorem~\ref{thm:sqrt-opt}}
For any $n\ge 1$, put $p=(1/n,1/n,\ldots,1/n)$
and $q=(0,0,\ldots,0)$. Then
$\tv{\Ber(p)-\Ber(q)}=1 - (1 - 1/n)^n
\ge1-\mathe\inv>0.63
$.
For $p',q'$, put $\gamma=1/n$
and define 
$p'_i=\frac12+\frac{\gamma}2$
and
$q'_i=\frac12-\frac{\gamma}2$.
Then by 
Theorem~\ref{thm:symm-ub},
we have 
$
\tv{\Ber(p')-\Ber(q')}
\le 
\nrm{p'-q'}_2
=n^{-1/2}
$.
This implies the $\sim\sqrt n$ gap between
the two variational distances.
Since $\nrm{p-q}_2
=
\nrm{p'-q'}_2
$, we are done. \qed

\subsection{Proof of Theorem~\ref{thm:symm-ub}}

For $p\in[0,1]^n$
and $q=1-p$,
put
$P=\Ber(p)$
and
$Q=\Ber(q)$.
\citet[Theorem 4]{ak-experts24}
implies
that
\beq
\sum_{
\omega\in\set{0,1}^n
}
\min\set{P(\omega),Q(\omega)}
&\ge& 
{
\prod_{i=1}^n 2\sqrt{p_iq_i}
}
\cdot 
\exp\paren{
-\frac12\sqrt{\sum_{i=1}^n 
\paren{\log\frac{p_i}{q_i}
}
^2}
}.
\eeq
We first
reparametrize
$p,q$ as 
$p_i=\frac12+\frac{x_i}2$
and
$q_i=\frac12-\frac{x_i}2$
for $x\in[0,1]^n$.
Thus, to prove the claim, it suffices to show that
\beq
1-
\prod_{i=1}^n \sqrt{1-x_i^2}
\cdot
\exp\paren{
-\frac12\sqrt{\sum_{i=1}^n 
\paren{\log\frac{
1+x_i
}{
1-x_i}
}
^2}
}
&\le&
\nrm{x}_2,
\qquad x\in[0,1]^n.
\eeq
Reparametrizing by $t_i=\tanh\inv x_i$,
the equivalent claim is
\beq
1-
\exp\paren{
-\sum_{i=1}^n\log\cosh t_i
-\sqrt{
\sum_{i=1}^n  t_i^2
}
}
&\le&
\sqrt{
\sum_{i=1}^n  
\tanh^2 t_i
},
\qquad
t_i\ge0.
\eeq
It suffices to prove the claim for $n=2$, as the general
case will follow the same argument.
To this end, we define
\beq
F(t_1,t_2)
&=&
\sqrt{
\sum_{i=1}^2  
\tanh^2 t_i
}
+
\exp\paren{
-\sum_{i=1}^2\log\cosh t_i
-\sqrt{
\sum_{i=1}^2  t_i^2
}
}
-1
\eeq
and compute
\beq
\ddel{}{t_1}F
&=&
\text{sech}(t_1)
\left(\frac{\mathe^{-\sqrt{t_1^2+t_2^2}}
\text{sech}(t_2)
\left(\sqrt{t_1^2+t_2^2} \tanh (t_1)+t_1\right)}{\sqrt{t_1^2+t_2^2}}+\frac{\tanh (t_1) \text{sech}(t_1)}{\sqrt{\tanh ^2 t_1+\tanh ^2 t_2}}\right)   
,
\eeq
the latter clearly
nonnegative.
The same holds for $\ddel{F}{t_2}$,
which implies that $F$ is minimized 
at $t_1=t_2=0$. Since 
$F(0,0)=0$, 
the proof is complete.
\qed

\section*{Acknowledgments}
We thank
Cl\'ement Canonne,
Ramon van Handel,
Mark Kozlenko,
Yury Makarychev,
Maxim Raginsky,
Mario Ullrich,
Roman Vershynin,
Rotem Zur,
and 
in
particular,
George Lowther, 
John (Tripp) Roberts,
Gregory Rosenthal,
and Thomas Steinke
for very useful discussions.
\appendix

\section{Auxiliary lemma}

The following result was proved by George Lowther,
answering a question we had posed. We provide a proof sketch for completeness.
\begin{lemma}[\citet{lowther-lemma}]
\label{lem:lowther}
Let $X_i$, $i\in[n]$, be independent symmetric random variables 
and $f:\R\to\R$ an increasing concave function with $f(0)=0$.
Define the random variables $Y=\abs{\sum_{i=1}^n X_i}$ and 
$Z=\sqrt{\sum_{i=1}^n X_i^2}$.
Then
\beqn
\label{eq:lowether}
\E f(Z) &\le& c\E f(Y),
\eeqn
where $c\le (1/\sqrt 2-1/4)\inv\approx 2.187$
and can conjecturally be improved to $2$.
\end{lemma}
\begin{proof}[Proof sketch]
Since the $X_i$ are symmetric, 
the expectations
in \eqref{eq:lowether}
are unchanged if each $X_i$ is 
replaced by $X_i'=\eps_i X_i$,
where the
$\eps_i$
are
independent
Rademacher random variables
with $\P(\eps_i=\pm1)=1/2$. Then a standard argument
lets us prove \eqref{eq:lowether} for $X_i'$
conditionally $X_i$,
from which the full claim will follow.
Thus, there is no loss of generality in taking
$\P(X_i=\pm a_i)=1/2$ for some $a_i\in\R$.
Further, we may assume $\nrm{a}_2=1$, since
any normalizing factor can be absorbed into $f$.

There is no loss of generality in assuming $f$ to be twice
differentiable, in which case we have the well known
integral representation
\beq
f(x) &=& \int_0^\infty \min\set{x,u}(-f''(u))\mathd u
.
\eeq
Thus, it suffices to consider $f$
of the form $f(t)=\min\set{t,u}$. Since we have assumed
$\nrm{a}_2=1$, the claim has been reduced to
\beq
c
\E\min\set{Y,u}
\ge f(Z)=\min\set{Z,u}.
\eeq
A 
straightforward
monotonicity argument shows that the case $u=1$ implies all the rest.

It remains to lower-bound 
$
\E\min\set{Y,1}
=
\E Y-\E\max\set{Y-1,0}
$.
Khintchine's inequality with the optimal constant
\citep{Haagerup1981} yields $\E Y\ge1/\sqrt 2$.
Since $y-1\le y^2/4$, we have
$
\E Y-\E\max\set{Y-1,0}
\le
\E[Y^2/2]=1/4
$.
\end{proof}

\bibliographystyle{plainnat}
\bibliography{refs}

\end{document}